\documentclass[a4,11pt]{article}
\usepackage[latin1]{inputenc}
\usepackage{epsfig}
\usepackage{amsmath}
\usepackage{amssymb}

\usepackage{amsthm}
\usepackage{color} 
\usepackage{mathrsfs}

\DeclareMathOperator{\arccot}{arccot}

\newtheorem{theo}{Theorem}[section]
\newtheorem{remark}{Remark}[section]

\newtheorem{coro}{Corollary}[theo]

\newcommand{\T}{\mathbb{T}}
\newcommand{\RR}{R}
\newcommand{\QQ}{Q}
\newcommand{\PP}{\mathcal{P}}
\newcommand{\RUC}{\mathcal{R}}

\DeclareMathAlphabet{\mathpzc}{OT1}{pzc}{m}{it}

\newcommand{\aR}{\mathfrak{a}}
\newcommand{\bR}{\mathfrak{b}}
\newcommand{\cR}{\mathfrak{c}}

\newcommand{\FFa}{\mathpzc{M}}
\newcommand{\FFb}{\mathpzc{N}}

\renewcommand{\Re}{\mathop{\rm Re}}
\renewcommand{\Im}{\mathop{\rm Im}}
\newcommand{\const}{\mathop{\rm const}}

\def\qed{\rule{1.0ex}{1.0ex} \medskip \medskip}

\def \dsp {\displaystyle}

\def\XXint#1#2#3{{\setbox0=\hbox{$#1{#2#3}{\int}$}
     \vcenter{\hbox{$#2#3$}}\kern-.5\wd0}}

\def\downbar#1{
\setbox10=\hbox{$#1$}
            \dimen10=\ht10 \advance\dimen10 by 2.5pt
            \ifdim \dimen10<15pt 
               \advance\dimen10 by -0.5pt
               \dimen11=\dimen10
               \advance\dimen10 by 2.5pt
               \lower \dimen11
            \else \lower \ht10 \fi
            \hbox {\hskip 1.5pt \vrule height \dimen10 depth \dp10}}
\def\upbar#1{
\setbox10=\hbox{$#1$}
            \dimen10=\ht10 \advance\dimen10 by \dp10 \advance\dimen10 by 2.5pt
            \ifdim \dimen10<15pt 
               \advance\dimen10 by 2pt \fi
            \aRise 2.5pt \hbox {\hskip -1.5pt \vrule height \dimen10}}

\begin{document}

\title{\Large{Complementary Romanovski-Routh polynomials: From orthogonal polynomials on the unit circle to Coulomb wave functions}\thanks
{This work was started while the first and the third authors were visiting Shanghai Jiao Tong University in the fall of 2016. Subsequently, this work was developed as part of the Ph.D. thesis of the second author, for which she had the grant 2017/04358-8 from FAPESP of Brazil. }}
\author
{
 {A. Mart\'{i}nez-Finkelshtein$^{a}$, L.L. Silva Ribeiro$^{b}$,  A. Sri Ranga$^{b}$\thanks{ranga@ibilce.unesp.br (corresponding author)}  and M. Tyaglov$^{c}$} \\[1ex]
  {\small $^{a}$Departamento de Matem\'{a}ticas, Universidad de Almer\'{i}a,}\\
  {\small 04120 Alamer\'{i}a, Espanha.}\\[1ex]
  {\small $^{b}$DMAp, IBILCE, UNESP - Universidade Estadual Paulista,} \\
  {\small 15054-000, S\~{a}o Jos\'{e} do Rio Preto, SP, Brazil.}\\[1ex]
  {\small $^{c}$School of Mathematical Sciences, Shanghai Jiao Tong University,} \\
  {\small Shanghai, China}\\[1ex]
}

\maketitle

\thispagestyle{empty}

\begin{abstract}
   We consider properties and applications of a sequence of polynomials known as complementary Romanovski-Routh  polynomials (CRR polynomials for short). These polynomials, which follow  from the Romanovski-Routh  polynomials or  complexified Jacobi polynomials, are known to be useful objects in the studies of the one-dimensional Schr\"{o}dinger equation and also the wave functions of quarks. One of the main results of this paper is to  show how the CRR-polynomials are related to a special class of orthogonal polynomials on the unit circle. As another main result, we have established their connection to a class of functions which are related to a subfamily of Whittaker functions that includes those associated with the Bessel functions and the regular Coulomb wave functions. An electrostatic interpretation for the zeros of CRR-polynomials is also considered.  
\end{abstract}

{\noindent}Keywords: Romanovski-Routh polynomials, Second order differential equations, Orthogonal polynomials on the unit circle, Para-orthogonal polynomials, Coulomb wave functions. \\

{\noindent}2010 Mathematics Subject Classification: 42C05, 33C45, 33C47.

\setcounter{equation}{0}
\section{Introduction}  \label{Sec-Intro}

The Romanovski-Routh polynomials are defined by the Rodrigues formula 
\begin{equation} \label{Eq-RR-Rodrigues}
    \RR_{n}^{(\alpha, \beta)}(x) = \frac{1}{\omega^{(\alpha, \beta)}(x)} \frac{d^{n}}{d x^{n}} \big[\omega^{(\alpha, \beta)}(x) (1+x^2)^{n}\big], \quad n \geq 1,  
\end{equation} 
(see \cite{Nikiforov-Uvarov-Book1988}), where $\omega^{(\alpha, \beta)}(x) = (1+x^2)^{\beta-1} (e^{-\arccot x})^{\alpha}$. These polynomials,   first appeared in   Routh \cite{Routh-PLMS1984},  were  rediscovered  by Romanovski \cite{Romanovski-AcadSciParis1929} in his work regarding probability distributions.  They  are found to be solutions of the second order differential equations 
\begin{equation} \label{Eq-DifEq-RR-Polynomials}
    (1+x^2) y^{\prime\prime}  + (2\beta x + \alpha) y^{\prime} - n(2\beta + n - 1) y = 0, \quad n \geq 1.  
\end{equation}
Since $\RR_{n}^{(\alpha, \beta)}(x) =  (-2i)^{n} n!\,P_{n}^{(\beta-1+\frac{i}{2}\alpha, \, \beta-1-\frac{i}{2}\alpha)} (ix)$, $ n \geq 0$, 
the polynomials $\RR_{n}^{(\alpha, \beta)}$ are also known as complexified Jacobi plynomials. The above connection formula between the polynomials $\RR_{n}^{(\alpha, \beta)}$ and the Jacobi polynomials $P_n^{(a,b)}$ can be directly verified from \eqref{Eq-RR-Rodrigues} and the Rodrigues formula (see, for example, \cite[Eq.\,(4.2.8)]{Ismail-Book2005}) for the Jacobi polynomials.  

The Romanovski-Routh polynomials $\RR_{n}^{(\alpha, \beta)}$ are not orthogonal polynomials in the usual sense. When $\beta$ is chosen to be negatively large enough (see, for example, \cite{RapWebCastKirc-CEJP2007}) then there holds a finite orthogonality. Precisely, 
\[
    \int_{-\infty}^{\infty} \RR_{n}^{(\alpha, \beta)}(x)\RR_{m}^{(\alpha, \beta)}(x)  \omega^{(\alpha, \beta)}(x)dx = 0 \quad \mbox{for} \quad  m \neq n, 
\]
only when $m+n -1 < - 2\beta$.

In this paper we will focus on a study of the so called {\em complementary Romanovski-Routh polynomials} (CRR polynomials for short). These polynomials  were defined in Weber \cite{Weber-CEJM-vol3} by considering a variation of the Rodrigues formula \eqref{Eq-RR-Rodrigues}. From the definition given in Weber \cite{Weber-CEJM-vol3}, the CRR  polynomials  are 
\begin{equation}\label{Eq-Q-Rodirigues}
     \QQ_{n}^{(\alpha, \beta)}(x) = \frac{(1+x^2)^{n}}{\omega^{(\alpha, \beta)}(x)} \frac{d^n}{dx^{n}} \omega^{(\alpha, \beta)}(x), \quad n \geq 1.
\end{equation}
The idea of looking at such complementary polynomials started in the work \cite{Weber-CEJM-vol2} of the same author.

As already observed in  \cite{RapWebCastKirc-CEJP2007}, one can  easily verify that $\QQ_{n}^{(\alpha, \beta)}(x) = \RR_{n}^{(\alpha, \beta-n)}(x)$, $n \geq 1$. 
Moreover, as shown also in \cite{RapWebCastKirc-CEJP2007} and \cite{Weber-CEJM-vol3}, the polynomials $\QQ_{n}^{(\alpha, \beta)}$  satisfy the three term recurrence formula 
\begin{equation} \label{Eq-Special-R2Type-RR-Weber}
    \QQ_{n+1}^{(\alpha, \beta)}(x) = [\alpha + 2(\beta-n-1)x]\QQ_{n}^{(\alpha, \beta)}(x) - n(-2\beta+n+1) (1+x^2) \QQ_{n-1}^{(\alpha, \beta)}(x), 
\end{equation} 
for $n \geq 1$, with $\QQ_{0}^{(\alpha, \beta)}(x) = 1$ and $\QQ_{1}^{(\alpha, \beta)}(x) = \alpha + 2(\beta-1)x$.  
It is also not difficult to verify from the differential equation \eqref{Eq-DifEq-RR-Polynomials} that 
\begin{equation}  \label{Eq-Hyper-Q}
    \QQ_{n}^{(\alpha, \beta)}(x) = (-2i)^{n}\big(\beta-n+i\frac{\alpha}{2}\big)_{n}\, _2F_1\Big(-n,2\beta-1-n;\,\beta-n+i\frac{\alpha}{2};\,\frac{1-ix}{2}\Big). 
\end{equation}

The CRR polynoimals are known to be (see \cite{RapWebCastKirc-CEJP2007}) important ingredients  in some studies concerning wave functions of quarks in accord with QCD (quantum chromodynamics) quark-gluon dynamics.  Moreover,  these  polynomials also play an important role in the studies of (one-dimensional) Schr\"{o}dinger  equation with hyperbolic Scarf potential.  

The main objective in the present work is to consider some  further properties  of these CRR  polynomials.  However, for convenience, we will view the CRR polynomials  using the modified  notation  $\PP_{n}(b; x)$, where $b = \lambda + i \eta$  and 
\begin{equation} \label{Eq-Modified-Notation}
    \PP_{n}(b; x) = \frac{(-1)^n }{2^n(\lambda)_{n}}\, \QQ_{n}^{(2\eta, -\lambda+1)}(x) = \frac{(-1)^n }{2^n(\lambda)_n} \RR_{n}^{(2\eta, -n-\lambda+1)}(x), \quad n \geq 1.
\end{equation}
We will also assume that $\lambda = \Re(b) > 0$ and  refer to the polynomials $\PP_{n}(b; .)$ as modified CRR polynomials or simply CRR polynomials. 

One of the reasons for adopting the notation $\PP_{n}(b; .)$, instead of $P_{n}^{(\lambda, \eta)}$ for example, is to avoid confusion with the notation used for Jacobi polynomials. 

The notation  $\{\PP_{n}(b;.)\}_{n \geq 0}$ is also somewhat interesting  to work with.   For example, in the three term recurrence formula \eqref{Eq-TTRR-for-CRR-polys} for $\{\PP_{n}(b;.)\}_{n \geq 0}$ given below, the sequence of coefficients $\{d_{n+1}^{(b)}\}_{n \geq 1}$ is exactly the same as the sequence of coefficients that appear in the three term recurrence formula for the monic Gegenbauer (i.e., ultrasperical) polynomials $\{\hat{C}_{n}^{(\lambda)}\}_{n\geq 0}$.  

In the following theorem  we have gathered  some of the basic properties of the CRR polynomials. 

\begin{theo} \label{Thm-Basic-Relations} For $b = \lambda + i \eta$, $\lambda > 0$, the complementary Romanovski-Routh polynomials $\PP_{n}(b;.)$ can be given by the hypergeometric expression 
\begin{equation*} \label{Eq-Pn-Hypergeometric}
     \PP_{n}(b;x) = \frac{(x-i)^{n}}{2^{n}} \frac{(2\lambda)_{n}}{(\lambda)_{n}}  \, _2F_1\Big(-n,b;\,b+\bar{b};\,\frac{-2i}{x-i}\Big), \quad n \geq 1.
\end{equation*}
They satisfy the three term recurrence 
\begin{equation} \label{Eq-TTRR-for-CRR-polys}
     \PP_{n+1}(b;x) = (x - c_{n+1}^{(b)}) \PP_{n}(b;x)  - d_{n+1}^{(b)} (x^2 + 1) \PP_{n-1}(b;x), \quad n \geq 1,
\end{equation}
with \ $\PP_0(b;x) = 1$ and  $\PP_{1}(b;x) = x - c_1^{(b)}$, where 
\begin{equation} \label{Coeffs-TTRR-for-Special-Pn}
    c_{n}^{(b)} = \frac{\eta}{\lambda+n-1} \quad \mbox{and} \quad d_{n+1}^{(b)} = d_{n+1}^{(\lambda)} = \frac{1}{4}\frac{n (2\lambda + n -1)}{(\lambda+n-1)(\lambda+n)}, \quad n \geq 1.
\end{equation}
Moreover, if $\lambda > 1/2$ then they also satisfy the orthogonality
\begin{equation} \label{Eq-Main-Orthogonality}
   \int_{-\infty}^{\infty} x^{m} \frac{\PP_{n}(b;x)}{(1+x^2)^n} \, \nu^{(\lambda, \eta)}(x) dx = \gamma_{n}^{(\lambda)} \delta_{m,n}, \quad m=0,1,\ldots, n,
\end{equation}
where 
\[
    \nu^{(\lambda, \eta)}(x) = \frac{2^{2\lambda-1} |\Gamma(b)|^2}{\Gamma(2\lambda-1)}   \frac{e^{\eta \pi}}{2\pi} \omega^{(2\eta, -\lambda+1)}(x) = \frac{2^{2\lambda-1} |\Gamma(b)|^2}{\Gamma(2\lambda-1)}   \frac{e^{\eta \pi}}{2\pi} \frac{(e^{-\arccot x})^{2\eta}}{(1+x^2)^{\lambda}}. 
\]
Here, $\gamma_{0}^{(\lambda)}  = \int_{-\infty}^{\infty} \nu^{(\lambda, \eta)}(x) dx = 1$ and $\gamma_{n}^{(\lambda)} = (1-\mathcal{L}_{n}^{(\lambda)})\gamma_{n-1}^{(\lambda)}$, $n \geq 1$, with  
\begin{equation} \label{Eq-maximal-ParamSeq}
     \mathcal{L}_{n}^{(\lambda)} = \frac{1}{2}\frac{2\lambda +n -2}{\lambda+n-1},  \quad n \geq 1 .
\end{equation}
\end{theo}

\begin{remark} Here, we have assumed $\arccot(x)$ to be a continuous function that decreases from $\pi$ to $0$ as $x$ increases from  $-\infty$ to $\infty$. 
\end{remark}

Part of the proof of Theorem \ref{Thm-Basic-Relations} directly follows from  some of the known results stated above.  However, we have given a proof of this theorem in Section  \ref{Sec-Conect-OPUC}  which follows as a consequence of some recently known results on orthogonal polynomials on the unit circle. 

From the three term recurrence \eqref{Eq-TTRR-for-CRR-polys} for  $\{\PP_{n}(b;.)\}_{n \geq 0}$ one can easily observe that 
the leading coefficient of $\PP_{n}(b;.)$ is positive.  Precisely, if $\PP_{n}(b;x) = \mathfrak{p}_{n}^{(b)} x^{n} + \mbox{lower order terms}$, then $\mathfrak{p}_{0}^{(b)} = 1$ and $\mathfrak{p}_{n}^{(b)} = (1- \ell_{n}^{(\lambda)}) \mathfrak{p}_{n-1}^{(b)}$, $ n \geq 1$, where 
\begin{equation} \label{Eq-minimal-ParamSeq}
   \quad 1- \ell_{n}^{(\lambda)} = \frac{(2\lambda)_{n}}{2^n(\lambda)_n}, \quad \ n \geq 1.
\end{equation}

The sequence $\{d_{n+1}^{(\lambda)}\}_{n \geq 1}$ in \eqref{Eq-TTRR-for-CRR-polys} is a so called positive chain sequence with the sequence $\{\ell_{n+1}^{(\lambda)}\}_{n \geq 0}$, as given above,  its minimal parameter sequence. That is,
\[
 (1-\ell_{n}^{(\lambda)}) \ell_{n+1}^{(\lambda)} = d_{n+1}^{(\lambda)}, \  n \geq 1, \quad \mbox{with} \quad  \ell_{1}^{(\lambda)} = 0 \ \  \mbox{and} \ \ 0 < \ell_{n}^{(\lambda)} < 1,  \   n \geq 2.  
\]
Any sequence $\{g_{n+1}\}_{n \geq 0}$ such that $ (1- g_{n})g_{n+1}= d_{n+1}^{(\lambda)}$, $n \geq 1$, with $ 0  \leq g_{1} < 1$ and $0  < g_{n} < 1$ for $n \geq 2$,  
can be referred to as a parameter sequence of the positive chain sequence $\{d_{n+1}^{(\lambda)}\}_{n \geq 1}$. When $1/2 \geq \lambda > 0$, the sequence $\{\ell_{n+1}^{(\lambda)}\}_{n \geq 0}$ is the only parameter sequence of $\{d_{n+1}^{(\lambda)}\}_{n \geq 1}$. 

When $\lambda > 1/2$,  the sequence  $\{\mathcal{L}_{n+1}^{(\lambda)}\}_{n \geq 0}$ given by \eqref{Eq-maximal-ParamSeq} is also such that 
\begin{equation} \label{Eq-Max-Parameter}
   (1- \mathcal{L}_{n}^{(\lambda)})\mathcal{L}_{n+1}^{(\lambda)} = d_{n+1}^{(\lambda)}, \   n \geq 1, \quad \mbox{with} \quad 0 < \mathcal{L}_{n}^{(\lambda)} < 1, \ n \geq 1. 
\end{equation}
Hence, when $\lambda > 1/2$, the sequence  $\{\mathcal{L}_{n+1}^{(\lambda)}\}_{n \geq 0}$ is also a parameter sequence of the positive chain sequence $\{d_{n+1}^{(\lambda)}\}_{n \geq 1}$.  It turns out $\{\mathcal{L}_{n+1}^{(\lambda)}\}_{n \geq 0}$ is the so called maximal parameter sequence of this positive chain sequence. For definitions and for many of the basic results concerning positive chain sequences we refer to \cite{Chihara-Book}. 

We can verify from \eqref{Eq-DifEq-RR-Polynomials} that the differential equation satisfied by $\PP_{n}(b;.)$ is 
\begin{equation} \label{Eq-DifEq}
   A(x)\,\PP_{n}^{\prime\prime}(b;x) - 2B(n,\lambda, \eta;x)\,\PP_{n}^{\prime}(b;x) + C(n,\lambda) \,\PP_{n}(b;x) = 0,
\end{equation}
where $A(x) = x^2 +1$, $B(n,\lambda, \eta;x) = (\lambda+n-1)x - \eta$,  \ $C(n,\lambda) = n(n-1+2\lambda)$.

Many other properties of $\PP_n(b;.)$ (i.e. of the polynomials $\QQ_n^{(\alpha, \beta)}$) have been explored in  \cite{Weber-CEJM-vol3} and \cite{RapWebCastKirc-CEJP2007}. Perhaps one of the most interesting and simplest of these properties is 
\begin{equation} \label{Eq-AppellProp}
   \frac{d\, \PP_{n}(b;x)}{dx} = n (1-\ell_{n}^{(\lambda)}) \PP_{n-1}(b;x), \quad n \geq 1,
\end{equation}
which can be verified from \eqref{Eq-Hyper-Q}. With this property, clearly  \eqref{Eq-TTRR-for-CRR-polys} and \eqref{Eq-DifEq} are equivalent statements. The results given in Section \ref{Sec-Appell}   of this paper are developed as a consequence of the property \eqref{Eq-AppellProp}.  

The contents  in the remaining sections of this paper are  as follows.  \\[-3ex]

\begin{itemize}

\item In Section \ref{Sec-Conect-OPUC}, as one of the main results of this paper, we show how the CRR polynomials $\PP_{n}(b; .)$ are related to a special class of orthogonal polynomials on the unit circle. A proof of Theorem \ref{Thm-Basic-Relations} is also given in this section. 

\item In Section \ref{Sec-Further-Prop-CRR}, we also give an eletrostatic interpretation for the zeros of CRR polynomials.

\item  Another main contribution of this paper, considered in  Section \ref{Sec-Appell}, is concerned with the generating function of the form  $e^{xw}\FFb(b;w)$ for the monic CRR polynomials. It turns out that the functions $\FFb(b;.)$ are closely related to the subfamily $M_{-i\eta, \lambda-1/2}$ of Whittaker functions  (see \cite{AndAskRoy-Book2000}). Thus, special cases of the function $\FFb(b;.)$ are also related to the Bessel functions and the regular Coulomb wave functions. We have referred to this subfamily of Whittaker functions as extended regular Coulomb wave functions.

\end{itemize}

\setcounter{equation}{0}
\section{Orthogonal polynomials on the unit circle}  \label{Sec-Conect-OPUC} 

We now consider the connection the CRR polynomials $\PP_{n}(b;.)$ have with the  polynomials $\Phi_{n}(b;.)$ which are orthogonal on the unit circle with respect to the probability measure 
\begin{equation} \label{Eq-UC-ProbMeasure}
    d \mu^{(b)}(e^{i\theta}) = \frac{4^{ \mathcal{R}e(b)} |\Gamma(b+1)|^2}{ \Gamma(b+\bar{b}+1)}\,  \frac{1}{2\pi}  [e^{\pi-\theta}]^{\mathcal{I}m(b)} [\sin^2(\theta/2)]^{\mathcal{R}e(b)} d \theta.
\end{equation}
Observe that the above measure presents a Fischer-Hartwig type singularity. The  monic orthogonal polynomials $\Phi_n(b;.)$ and the associated orthogonality relation, which exist for $\lambda > -1/2$,   are explicitly  given by (see \cite{ASR-PAMS2010})  
\begin{equation}\label{Eq-OPUC-OrthoRelation}
  \begin{array}l
    \dsp  \Phi_{n}(b;z) = \frac{(2\lambda+1)_{n}}{(b+1)_{n}} \, _2F_1(-n,b+1;\,b+\bar{b}+1;\,1-z), \quad n \geq 0, 
  \end{array}
\end{equation}
 and,   for $n, m = 0, 1, 2 , \ldots$, 
\[
   \int_{0}^{2\pi}  \overline{\Phi_{m}(b;e^{i\theta})}\,\Phi_{n}(b;e^{i\theta})\, d \mu^{(b)}(e^{i\theta}) = \frac{(2\lambda+1)_{n}\,n!}{|(b+1)_{n}|^2} \delta_{m,n}.
\]

For general information on orthogonal polynomials on the unit circle we refer to the monographs  Szeg\H{o} \cite{Szego-book}, Simon \cite{Simon-Book-p1},  Simon \cite{Simon-Book-p2} and Ismail \cite{Ismail-Book2005}. 

Let the polynomials $\{\RUC_{n}(b;.)\}_{n \geq 0}$ be such that
\begin{equation} \label{Eq-Transformation-Rn-Pn}
     \RUC_{n}(b;\zeta) = \frac{2^{n}}{(x-i)^{n}} \PP_{n}(b;x), \quad  n \geq 0,
\end{equation}
where $\zeta  = (x+i)/(x-i)$ or equivalently $x = i(\zeta+1)/(\zeta-1)$. This transformation maps the real line $(-\infty, \infty)$ onto the cut unit circle $\T\backslash 1 := \{\zeta=e^{i\theta}\!\!: \, 0 < \theta < 2\pi \}$.  
With known results about the polynomials $\RUC_{n}(b;.)$ on the unit circle, which we will point out as needed,  we now look into a  proof of Theorem \ref{Thm-Basic-Relations}. 

\begin{proof}[Proof of Theorem \ref{Thm-Basic-Relations}]   

From \eqref{Eq-DifEq} and \eqref{Eq-Transformation-Rn-Pn}, it is not difficult to see that 
\begin{equation*} \label{Circle-DifEq}
   z(1-z)\frac{d^2 \RUC_{n}(b;\,z)}{dz^2} - [b+\bar{b} - (-n+b+1)(1-z)]\frac{d \RUC_{n}(b;\,z)}{dz} + n b \RUC_{n}(b;\,z) = 0,
\end{equation*}
for $n \geq 1$.  From this  differential equation  we can easily identify the polynomials $\RUC_{n}(b;\, .)$ to be 
\begin{equation} \label{Eq-Rn-HyperForm}
     \RUC_{n}(b;\, z) = \frac{(2\lambda)_{n}}{(\lambda)_{n}} \, _2F_1(-n,b;\,b+\bar{b};\,1-z), \quad n \geq 0. 
\end{equation}
This result, together with \eqref{Eq-Transformation-Rn-Pn}, gives the proof of the hypergeometric expression given in Theorem \ref{Thm-Basic-Relations}. However, the hypergeometric expression in Theorem \ref{Thm-Basic-Relations} can also be obtained from \eqref{Eq-Hyper-Q}  using two transformations of Pfaff, namely the ones given by (2.2.6) and (2.3.14) in  \cite{AndAskRoy-Book2000}. 

Now, if the three term recurrence in Theorem \ref{Thm-Basic-Relations} is true then one must also have    
\begin{equation} \label{TTRR-for-Special-Rm}
     \RUC_{n+1}(b;\, z) = \left[(1 + i\,c_{n+1}^{(b)})z + (1 - i\,c_{n+1}^{(b)})\right] \RUC_{n}(b;\, z) - 4 d_{n+1}^{(b)} z \RUC_{n-1}(b;\,z),
\end{equation}
for $n \geq 1$, with $\RUC_0(b;\, z) = 1$ and $\RUC_1(b;\, z) = (1 + i\,c_1^{(b)})z + (1 - i\,c_1^{(b)})$, where the coefficients $c_{n}^{(b)}$ and $d_{n+1}^{(b)}$ are as in \eqref{Coeffs-TTRR-for-Special-Pn}. 

From results given in \cite{Costa-Felix-ASR-JAT2013} and \cite{ASR-PAMS2010} we also find that the polynomials $\RUC_{n}(b;.)$ given by \eqref{Eq-Rn-HyperForm} actually satisfy the above three term recurrence.  Thus, from \eqref{Eq-Transformation-Rn-Pn} the three term recurrence relation in  Theorem \ref{Thm-Basic-Relations} is also verified. Once again, the three term recurrence relation in Theorem \ref{Thm-Basic-Relations} can also be easily derived from the three term recurrence relation \eqref{Eq-Special-R2Type-RR-Weber} for the polynomials $\QQ_{n}^{(2\eta, -\lambda+1)}$.  

The polynomials $\RUC_{n}(b;\cdot)$ (see \cite{ASR-PAMS2010}) are also the para-orthogonal  polynomials 
\begin{equation} \label{Eq-as-ParaOrth}
   \RUC_{n}(b;\, z) = \frac{(b)_{n}}{(\lambda)_{n}} \big[ z\Phi_{n-1}(b; z) + \frac{(\overline{b})_{n}}{(b)_{n}} \Phi_{n-1}^{\ast}(b; z)\big], \quad n \geq 1, 
\end{equation}
associated with the orthogonal polynomials on the unit circle $\Phi_{n}(b;.)$ with respect to the measure $\mu^{(b)}$ given above.

The polynomials $\RUC_{n}(b;\cdot)$ and the associated orthogonal polynomials $\Phi_{n}(b; \cdot)$, in addition to have been studied in \cite{ASR-PAMS2010}, they have been used as examples in a sequence of papers  \cite{Brac-ASR-Swami-ANM2016, CFB-MF-ASR-DOV-JAT2008, Castillo-Costa-ASR-Veronese-JAT2014, Costa-Felix-ASR-JAT2013, Dim-ASR-MN2013, Ismail-ASR-2016, MFinkelshtein-ASR-Veronese-2016},  without knowing anything about their connection to the CRR polynomials. The results obtained in \cite{Ismail-ASR-2016} are focused on the three term recurrence of the type \eqref{Eq-TTRR-for-CRR-polys} and the associated generalized eigenvalue problem (with these respect see also \cite{IsmailMasson-JAT1995} and \cite{Zhedanov-JAT1999}).  From what we can observe from \cite[p.\,304]{Askey-SzegoCP-1982}, the polynomials $\RUC_{n}(b;\cdot)$ and $\Phi_{n}(b; \cdot)$ were also have been observed as belonging to a class of hypergeometric biorthogonal polynomials. But, again in \cite{Askey-SzegoCP-1982}, no such connection  to the Romanovski-Routh polynomials has been mentioned. 

However, the connection between the  circular Jacobi polynomials  and a subfamily  of  the CRR polynomials  is somewhat known in the literature (see \cite{Witte-Forrester-NLinearity2000}). We recall that the circular Jacobi polynomials are the subclass of the polynomials $\Phi _{n}(b; \cdot)$ in which  $\Im(b) = \eta = 0$.  

Even though the polynomials $\RUC_{n}(b;.)$ and their associated CRR polynomials $\PP_{n}(b;.)$ are defined for $\Re(b) = \lambda > 0$,  the orthogonal polynomials $\Phi_{n}(b;.)$ themselves can be considered for $\lambda > -1/2$.   

We now consider the reproducing kernels  
\[
  \begin{array}{rl}
    K_{n}(b; z, w) 
    & = \dsp \frac{\overline{\varphi_{n+1}(b;w)}\,\varphi_{n+1}(b;z) - \overline{\varphi_{n+1}^{\ast}(b;w)}\varphi_{n+1}^{\ast}(b;z)}{\overline{w} z-1}, \quad n \geq 0.
  \end{array}
\]
Here, $\varphi_{n}(b;.)$ are the orthonormal version of $\Phi_{n}(b; \cdot)$.

The above kernels are also known as Christoffel-Darboux kernels or simply CD Kernels (see, for example, \cite{Simon-Book-p1}). For any fixed $w$, the kernel $K_{n}(b; z, w)$ is a polynomial in $z$ of degree $\leq n$ and, in particular, if $|w| \geq 1$ then it is a polynomials in $z$ of exact degree $n$. 

When $\lambda > 1/2$, the polynomials $\RUC_{n}(b;.)$ are also modified kernel polynomials of the orthogonal polynomials $\Phi_n(b-1;.)$, evaluated at $w = 1$. That is, 
\[
    \RUC_{n}(b;\, z) = \xi_{n}^{(b-1)}\,K_{n}(b-1; z, 1), \quad n \geq 0,
\]
where
$\xi_{0}^{(b-1)} = \int_{\T} d \mu^{(b-1)}(\zeta)$ and $\xi_{n}^{(b-1)} = \xi_{0}^{(b-1)} \prod_{j=1}^{n} (1-\mathcal{L}_{j}^{(\lambda)})$, $n \geq 1$. 
Here,  $\{\mathcal{L}_{n}^{(\lambda)}\}_{n \geq 1}$ is the sequence given by \eqref{Eq-maximal-ParamSeq} and, since the measure $\mu^{(b-1)}$ being a probability measure, $\xi_{0}^{(b-1)} = 1$.  
Known results on Kernel polynomials also give the orthogonality (see \cite{Costa-Felix-ASR-JAT2013})
\[
    \int_{\mathbb{T}} \zeta^{-k}\RUC_{n}(b;\zeta)  (1-\zeta^{-1}) d \mu^{(b-1)}(\zeta)   = 0,  \quad 0 \leq k \leq n-1. 
\] 
From this, by using the transformation \eqref{Eq-Transformation-Rn-Pn} we obtain the orthogonality  \eqref{Eq-Main-Orthogonality} 
in Theorem \ref{Thm-Basic-Relations},  where to be precise $\nu^{(\lambda, \eta)}(x) dx = -d \mu^{(b-1)}(\zeta)$.

Finally, to obtain the formula for  $\gamma_{n}^{(\lambda)}$ in Theorem \ref{Thm-Basic-Relations}, from  \eqref{Eq-TTRR-for-CRR-polys} we have 
\[
    \frac{x^{n-1}\PP_{n+1}(b;x)}{(x^2+1)^{n}}  = (x-c_{n+1}^{(b)})\frac{x^{n-1}\PP_{n}(b;x)}{(x^2+1)^{n}}  - d_{n+1}^{(\lambda)}\frac{x^{n-1}\PP_{n-1}(b;x)}{(x^2+1)^{n-1}},  
\]
for $n \geq 1$. Hence, integration with respect to $\nu^{(\lambda, \eta)}$ gives $\gamma_{n+1}^{(\lambda)} = \gamma_{n}^{(\lambda)} - d_{n+1}^{(\lambda)}\gamma_{n-1}^{(\lambda)}$, $n \geq 1$,  
which can be  written in the alternative form 
\[
   \frac{\gamma_{n}^{(\lambda)}}{\gamma_{n-1}^{(\lambda)}} \Big(1 - \frac{\gamma_{n+1}^{(\lambda)}}{\gamma_{n}^{(\lambda)}}\Big) = d_{n+1}^{(\lambda)}, \quad n \geq 1. 
\]
Thus from \eqref{Eq-Max-Parameter}, what one has to verify is $\gamma_{1}^{(\lambda)}/\gamma_{0}^{(\lambda)} = (1- \mathcal{L}_{1}^{(\lambda)}) = (2\lambda)^{-1}$. 
Clearly,  $\gamma_{0}^{(\lambda)} = \int_{-\infty}^{\infty}  \nu^{(\lambda, \eta)}(x) dx = \int_{\T} d\mu^{(b-1)}(\zeta) = 1$. However,  
\[
   \begin{array}{rl}
      \dsp \gamma_{1}^{(\lambda)} = \int_{-\infty}^{\infty} x\frac{\PP_{1}(b;x)}{(1+x^2)} \, \nu^{(\lambda, \eta)}(x) dx 
          & \dsp= \frac{1}{4}\int_{\T} \frac{\zeta+1}{\zeta} \RUC_{1}(b;\zeta) d \mu^{(b-1)}(\zeta) \\[2ex]
          &= (1 + i c_{1}^{(b)})\mu^{(b-1)}_{-1}  +2 + (1 - i c_{1}^{(b)})\mu^{(b-1)}_{1} . 
   \end{array}
\]
Thus, using the expression for $c_1^{(b)}$ in Theorem \ref{Thm-Basic-Relations} together with $\mu^{(b-1)}_{1} = \overline{\mu}^{(b-1)}_{-1} = (-b+1)/\overline{b}$, we find $\gamma_{1}^{(\lambda)}  = (2\lambda)^{-1}$. This completes the proof of Theorem \ref{Thm-Basic-Relations}.   \hfill
\end{proof}

\setcounter{equation}{0}
\section{Electrostatic interpretation for the zeros of CRR polynomials }   \label{Sec-Further-Prop-CRR}

Consider the electrostatic field which obeys the logarithmic potential law composed of two fixed negative charges of size $\lambda_{m}$  at $i$ and $-i$,  and $m$ movable positive unit charges along the x-axis.  In addition, assume that  there is also a background energy filed of type $\arctan$ effecting the movable charges. To be precise, we consider the electrostatic field where the  energy $E = E(x_1, x_2, \ldots, x_{m})$ is given by 
\[
   E = \sum_{1 \leq j < i \leq m} \ln \frac{1}{|x_{j}- x_{i}|} -\frac{\lambda_{m}}{2} \sum_{j=1}^{m}\Big[ \ln \frac{1}{|x_{j} - i|} + \ln \frac{1}{|x_{j} + i|} \Big] - \eta \sum_{j=1}^{m} \arctan(x_{j}).   
\]
With $\lambda_n$ taken to be large enough, we need to find the set of locations of the movable charges that minimizes this energy. 

\begin{theo} 
Let $\lambda_{m} = \lambda+ m -1$. Then the set of values $x_1^{(m)}(b), x_2^{(m)}(b), \ldots, \linebreak x_{m}^{(m)}(b)$ that minimizes the the energy $E(x_1, x_2, \ldots, x_{m})$ are the zeros of $\PP_{m}(b;.)$.
\end{theo}

\noindent {\bf Proof}. The problem of minimizing $E$ is equivalent to maximizing the positive function  
\[
    F = F(x_1, x_2, \ldots, x_{m}) = e^{-2E(x_1, x_2, \ldots, x_{m})}.  
\]
Clearly,
\[
    F = \prod_{1 \leq j < i \leq m}(x_{j} - x_{i})^2 \  \prod_{j=1}^{m} (1+x_{j}^2)^{-\lambda_{m}} \  \prod_{j=1}^{m} e^{2\eta \arctan{x_{j}}}.
\]
With $\lambda_{m} > m-1$, one can easily verify that $F$ is bounded from above. We write 
\[
   F = f_{m}(\hat{x}\backslash x_{k})\ \prod_{j \neq k} (x_{k}-x_{j})^2 \ \ (1+x_{k}^2)^{-\lambda_{m}}\ e^{2\eta \arctan{x_{k}}},
\]
where $f_{m}(\hat{x}\backslash x_{k})$ is the part of $F$ that is independent  of $x_{k}$. With  $r_{k}(x) = r(x)/\linebreak (x-x_{k})$, where $r(x) = \prod_{j=1}^{m}(x-x_{j})$,  we can also write 
\[
    F = f_{m}(\hat{x}\backslash x_{k})\ r_{k}^2(x_{k}) \  (1+x_{k}^2)^{-\lambda_{m}}\ e^{2\eta \arctan{x_{k}}},
\]
and our problem is now to find a polynomial $r(x)$ of degree $m$ that maximizes $F$. 

Differentiating $F$ with respect to $x_{k}$ gives 
\[ 
  \begin{array}{rl}
  \frac{\partial F}{\partial x_{k}} = \dsp & f_{m}(\hat{x}\backslash x_{k})  \Big[2r_{k}(x_{k})r_{k}^{\prime}(x_{k})\, (1+x_{k}^2)^{-\lambda_{m}}\, e^{2\eta \arctan{x_{k}}} \\[2ex]
  & \dsp \qquad \qquad  \quad  - 2\lambda_{m} x_{k}(1+x_{k}^2)^{-\lambda_{m}-1} \,r_{k}^2(x_{k})\,e^{2\eta \arctan{x_{k}}} \\[1ex]
  & \dsp \qquad  \qquad  \qquad \qquad + \frac{2\eta}{1+x_{k}^{2}}e^{2\eta \arctan{x_{k}}}\, r_{k}^2(x_{k})\, (1+x_{k}^2)^{-\lambda_{m}} \Big] \\[3ex]
  = & \dsp f_{m}(\hat{x}\backslash x_{k})\ r_{k}(x_{k}) (1+x_{k}^2)^{-\lambda_{m}-1}e^{2\eta \arctan{x_{k}}} \\[2ex] 
   & \dsp \qquad \qquad \qquad \times \Big[ 2(1+x_{k}^{2}) r_{k}^{\prime}(x_{k}) - 2(\lambda_{m} x_{k} - \eta) r_{k}(x_{k})  \Big]
   \end{array}
\]
Using $r^{\prime}(x_{k}) = r_{k}(x_{k})$ and $r^{\prime\prime}(x_{k}) = 2r_{k}^{\prime}(x_{k})$, we then have 
\[ 
 \begin{array}{rl}
   \frac{\partial F}{\partial x_{k}} & = \dsp f_{m}(\hat{x}\backslash x_{k})\ r_{k}(x_{k}) (1+x_{k}^2)^{-\lambda_{m}-1}e^{2\eta \arctan{x_{k}}} \\[1.5ex]
  & \dsp \qquad \qquad  
  \times  \big[ (1+x_{k}^{2}) r^{\prime\prime}(x_{k}) - 2(\lambda_{m} x_{k} - \eta) r^{\prime}(x_{k})  \big]  .
 \end{array}
\]
The expression $s(x)=(1+x^2) r^{\prime\prime}(x) - 2(\lambda_{m} x - \eta) r^{\prime}(x)$ is a polynomial of degree $m$.  We must choose $r(x)$ such that  at the maximum of $F$ the polynomial $s(x)$ also vanishes at the zeros of $r$. That is, $s(x) = \const \times r(x)$. 

Thus, if we take $\lambda_{m} = \lambda + m -1$, where $\lambda > 0$, then from the differential equation \eqref{Eq-DifEq} we find $r(x) = \const \times P_{m}(b;x)$ and $x_{k}$ are the zeros of $P_{m}(b;x)$. 
\hfill \qed

\vspace{-1ex}

\setcounter{equation}{0}
\section{Generating functions }  \label{Sec-Appell}

Generating functions have been known to be an important tool in the theory of special functions.    The following  generating function for the CRR polynomials $\QQ_{n}^{(2\eta, -\lambda+1)}$ is given in \cite[Thm.\,1.9] {Weber-CEJM-vol3}:
\[
    \frac{(x^2+1)^{\lambda} e^{2\eta \arccot(x)}}{ \big[1 + [x + w(x^2 +1)]^{2}\big]^{\lambda} e^{2\eta \arccot[x + w(x^2 +1)]}} = \sum_{n=0}^{\infty} \ \QQ_{n}^{(2\eta, -\lambda+1)}(x) \frac{w^n}{n!}.
\]

From a simplification of the above left hand side  and then using \eqref{Eq-Modified-Notation} we can state the following result. 

\begin{theo} \label{Thm-GenFunc-1} For $b = \lambda +i \eta$ and  $\lambda > 0$, 
\[
    \frac{e^{2\eta \arccot(x)}}{[(xw - 1)^2 + w^2]^{\lambda}e^{2\eta \arccot[x - w(x^2 +1)]}}  =   \sum_{n=0}^{\infty}  (2\lambda)_{n} \widehat{\PP}_{n}(b; x) \frac{w^{n}}{n!},
\]
where $\widehat{\PP}_{n}(b;.)$ are the monic CRR polynomials given by 
\begin{equation} \label{Eq-Monic-CRR-Polys}
        \widehat{\PP}_{n}(b;x) = \frac{1}{\mathfrak{p}_{n}^{(b)}}\PP_{n}(b;x) = \frac{2^{n} (\lambda)_{n}}{(2\lambda)_{n}} \PP_{n}(b;x),  \quad n \geq 1.
\end{equation}
\end{theo}

\subsection{Generating function as an Appell sequence} \label{Subsec-Appell}

In order to see the importance of the new generating  function that we obtain for  $\{\widehat{\PP}_{n}(b; x) \}_{n \geq 0}$, we consider the functions $\FFa(b;w)$ and $\FFb(b;w)$ given by 
\begin{equation} \label{Eq-ERCWfun-def}
       \FFa(b;w) = \mathfrak{C}(b)\, w^{\lambda}\, \FFb(b; w) \quad \mbox{and} \quad \FFb(b;w) = e^{-iw} \,_1F_1(b;\,b + \overline{b};\,2iw), 
\end{equation}
where   
\begin{equation} \label{Eq-Gamow-Constant}
      \mathfrak{C}(b) =  2^{\lambda-1}  e^{\pi\eta/2} \frac{|\Gamma(b)|}{\Gamma(2\lambda)}
\end{equation}
and,  as we have assumed so far,  $b = \lambda + i \eta$ and $\lambda > 0$.  Here, the notation  $_1F_1$ denotes Kummer's confluent hypergeometric function. 

Clearly, from the $_1F_1$ hypergeometric representation  
\begin{equation} \label{Eq-G-to-Whittaker}
   \FFa(b;w) =   (i2)^{-\lambda}\mathfrak{C}(b)\,  M_{-i\eta, \lambda-1/2}(2iw) , \quad  \lambda > 0, 
\end{equation} 
where $M_{-i\eta, \lambda-1/2}$ is a subclass of the Whittaker functions (see, for example, \cite[Pg.\,195]{AndAskRoy-Book2000}). 
The function $\FFa(b;w)$ becomes familiar if one considers the alternative notations 
\begin{equation} \label{Eq-Alternative-Notations}
     \FFa(\overline{b};w)=F_{\lambda-1}(\eta,w) \quad \mbox{and} \quad \mathfrak{C}(\overline{b}) =  C_{\lambda-1}(\eta).
\end{equation}
When $\lambda$ takes positive integer values, the resulting functions $F_{L}(\eta,w)$, $L = 0,1,  \ldots$, are the so called regular Coulomb wave functions. Precisely, with our  definitions of $\FFa(b; w)$  and $\FFb(b; w)$  the regular Coulomb wave functions can be given by 
\begin{equation} \label{Eq-Connection-Coulomb}
     F_{L}(\eta, w) =  \FFa(L+1-i\eta;w) = C_{L}(\eta) w^{L+1} \FFb(L+1-i \eta; w),  \quad L = 0, 1, 2,  \ldots. 
\end{equation}
It is known that the regular Coulomb wave functions satisfy the differential equation 
\begin{equation} \label{Eq-EDO-Coulomb}
   F_{L}^{\prime\prime}(\eta, w) + \Big[1 - \frac{2\eta}{ w} - \frac{L(L+1)}{ w^{2}}\Big]\, F_{L}(\eta, w) = 0, 
\end{equation}
where $F_{L}^{\prime\prime}(\eta, w) = d^2F_{L}(\eta, w)/dw^2$. Moreover, it is also known that 
\begin{equation*} \label{Eq-TTRR-Coulomb}
   F_{L+1}(\eta, w) = \frac{(2L+1)}{ L\,|L+1+i\eta|}  \Big[\frac{L(L+1)}{w} + \eta\Big] F_{L}(\eta, w)  - \frac{(L+1)  |L+i\eta|}{ L\, |L+1+i\eta|}\, F_{L-1}(\eta, w),
\end{equation*} 
which holds for $ L = 1, 2, 3, \ldots $. This three term recurrence relation associated with the Coulomb wave functions  was first given by Powel in \cite{Powel-PhyRev1947}.

As stated  in \cite{Gautschi-SIAMRev1967}, the Coulmb wave  functions are of great importance in the study of nuclear interactions. They arise when Schr\"{o}dinger's  equation for a charged particle in the Coulomb field of a fixed charge is separated in polar coordinates.  For some of the earliest studies  on these functions we cite \cite{Shepanski-Butler-NucPhy1956} and references therein. 

Numerical evaluation of regular Coulomb wave functions have been the subject of many contributions including  \cite{Abromowitz-Stegun-1972, Froberg-RevMPhys1955, Gautschi-SIAMRev1967, Meligy-NucPhy1958, Shepanski-Butler-NucPhy1956}.  Except for  \cite{Meligy-NucPhy1958, Shepanski-Butler-NucPhy1956}, they are mainly based on the use of the above three term recurrence relation.  Moreover,  the derivation of some of the basic properties of the zeros of these regular Coulomb wave functions and also the evaluation of these zeros  have been based on an  eigenvalue problem that follows from this  three term recurrence relation  (see \cite{Ikebe-MCOM1975, Miya-Kiku-Cai-Ikebe-MCOM2001}). 

By examining the differential equations,  the three term recurrence relation and also the associated eigenvalue problems satisfied by  the regular Coulomb wave functions, it is somewhat  evident  that the function $F_{\lambda}(\eta,w) = \FFa(\overline{b}+1;w)$, obtained by extending the integer parameter $L$  to the real parameter $\lambda$, will preserve many of the properties satisfied by the Coulomb wave function $F_{L}(\eta,w)$, in particular with regards to the zeros. This is clearly true in the case $\eta = 0$ and the resulting functions lead to the Bessel functions.  With such a knowledge,  the author of \cite{Baricz-JMAA2015} looks at some Tur\'{a}n type inequalities associated with these extended regular Coulomb wave functions $F_{\lambda}(\eta,w)$ and obtains also some information regarding the zeros of these functions. With a different  objective,  the authors of  \cite{Stamp-Stovic-JMAA2014} study the orthogonal polynomials that follow from the extended three term recurrence relation, which they call orthogonal polynomials associated with Coulomb wave functions. For some other contributions regarding the functions $F_{\lambda}(\eta,w)$ for non-integer values of $\lambda$, and even complex values of the parameter $\lambda$, we cite, for example,  \cite{DeanoSeguraTemme-JCAM2010, Michel-CompPhyComm2007} and references therein.  

In view of the above observations, in the present manuscript we will refer to the functions $\FFa(b;w) =F_{\lambda-1}(-\eta,w)$  as {\em extended regular Coulomb wave} functions (ERCW functions for short). 

From the recurrence for the regular Coulomb wave functions there follows  
\begin{equation} \label{Eq-TTRR-Fb}
   \FFa(b+2;w) = \frac{(2\lambda+1)}{ \lambda\,|b+1|}  \Big[\frac{\lambda(\lambda+1)}{w} - \eta\Big] \FFa(b+1;w) - \frac{(\lambda+1)  |b|}{ \lambda\, |b+1|}\, \FFa(b;w),
\end{equation} 
which holds for $ \lambda > 0$. This result can be  verified from  \eqref{Eq-G-to-Whittaker} and from well known contiguous  relations (see \cite[pg.\,27]{Slater-book}) satisfied by  Whittaker functions. 

The following theorem gives the the role played by the ERCW function $\FFa(b;w)$ as part of a  generating function for the   monic CRR polynomials $\widehat{\PP}_{n}(b;x)$.

\begin{theo} \label{Thm-Appell-Expansion} 
Let $b = \lambda + i\eta$, where $\lambda > 0$. Then the sequence $\{\widehat{\PP}_{n}(b;.)\}_{n \geq 0}$ of monic complementary Romanovski-Routh polynomials is an Appell sequence and satisfy  
\begin{equation} \label{Eq-Appell-GenFunc}
     \frac{1}{\mathfrak{C}(b)} e^{xw} w^{-\lambda} \FFa(b;w) = e^{xw}\,\FFb(b; w) = \sum_{n=0}^{\infty} \widehat{\PP}_{n}(b;x)\, \frac{w^{n}}{n!}. 
\end{equation}  

\end{theo}

\noindent {\bf Proof}.  The property \eqref{Eq-AppellProp}  of the CRR polynomials $\PP_{n}(b;x)$ shows us that the monic polynomial sequence $\{\widehat{\PP}_{n}(b;x)\}_{n \geq 0}$ is an Appell \cite{Appell-1880} sequence.  Thus, there holds a generating function of the form $e^{xw} F(w)$. We need to prove  $F(w) = \FFb(b;w)$.  

By setting  $F(w) = e^{-iw} H(w)$, where $ H(w) = \sum_{j=0}^{\infty} h_{j}^{(b)} w^{j}$, 
therefore one needs to prove that \eqref{Eq-Appell-GenFunc} holds if  $H(w) =  \,_1F_1(b;\,b + \overline{b};\,2iw)$.   That is, 
\begin{equation} \label{Eq-Hipergeo-Appell}
        h_{j}^{(b)}  = \frac{(b)_{j}\, (2i)^{j} }{(2\lambda)_{j}\, j!}, \quad j \geq 0.
\end{equation}
We have   
\[
     e^{xw}\,F(w) = e^{(x-i)w} \sum_{j=0}^{\infty} h_{j}^{(b)} w^{j} = \Big(\sum_{k=0}^{\infty} \frac{(x-i)^k}{k!}\,w^k\Big)\, \Big(\sum_{j=0}^{\infty} h_{j}^{(b)} w^{j}\Big). 
\]
Hence, 
\[
    e^{xw}\,F(w)  = \sum_{n=0}^{\infty} \Big( \sum_{l=0}^{n} h_{n-l}^{(b)} \frac{(x-i)^{l}}{l!}\Big) w^{n}. 
\]
From the hypergeometric expression for $\PP_{n}(b;.)$ in Theorem  \ref{Thm-Basic-Relations} we can easily verify that if 
\[
    \sum_{l=0}^{n} h_{n-l}^{(b)} \frac{(x-i)^{l}}{l!} = \frac{1}{n!}\widehat{P}_{n}(b;x), \quad n \geq 0,  
\]
then $h_{j}^{(b)}$ are as in  \eqref{Eq-Hipergeo-Appell}.  This completes the proof of the Theorem. \hfill \qed

\begin{remark}
The function $H(w)$ considered in the proof of Theorem \ref{Thm-Appell-Expansion} is  a $\,_1F_1$ confluent hypergeometric function and hence it is an entire function. Thus, the function $F(w) = \FFb(b;w)$ and also, for each  $x$, the generating function  $e^{xw}\FFb(b; w)$ are entire functions. Thus, the right hand side of  \eqref{Eq-Appell-GenFunc} is absolutely convergent for all $w$ on the complex plane.    
\end{remark}

By letting $b = \lambda + i \eta = 1$ in Theorem \ref{Thm-Appell-Expansion},  it is not difficult to verify the following. 

\begin{coro} 
\[  
        e^{xw} w^{-1} \sin(w) =  e^{xw}\,\FFb(1;w) = \sum_{n=0}^{\infty} \widehat{\PP}_n(1;x) \frac{w^n}{n!} . 
\]
\end{coro}


The ECWF function $\FFa(b;w)$, when $\Im(b) = \eta = 0$, is related to the Bessel function of order $\lambda-1/2$. To be precise,
\[
    \frac{2\, \Gamma(2\lambda)}{\Gamma(\lambda)} (2w)^{-\lambda} \FFa(\lambda;w) = \FFb(\lambda;w) = \Gamma(\lambda+1/2) \Big(\frac{w}{2}\Big)^{-\lambda + 1/2} J_{\lambda-1/2}(w). 
\] 
For more information on  Bessel functions see, for example,  \cite{Abromowitz-Stegun-1972}, \cite{AndAskRoy-Book2000} and \cite{Watson-Book}.  We can now state the following. 

\begin{coro}  \label{Coro-Bessel-Expansion}
For $\alpha > -1/2$ the following expansion formula holds with respect to the Bessel function $J_{\alpha}(w)$: 
\[
    e^{xw} J_{\alpha}(w) = \frac{1}{\Gamma(\alpha+1)} \big(\frac{w}{2}\big)^{\alpha} \sum_{n=0}^{\infty} \widehat{\PP}_n(\alpha+1/2;x) \frac{w^n}{n!} .
\]

\end{coro}

With the connection \eqref{Eq-Connection-Coulomb} with the regular Coulomb wave functions we can state:  

\begin{coro} \label{Coro-Coulomb-Expansion}
The following expansion formula holds with respect to the regular Coulomb wave function $F_{L}(\eta,w)$:
\[
    e^{xw} \FFa(L+1-i\eta;w) = e^{xw} F_{L}(\eta,w) = C_{L}(\eta)\, w^{L+1} \sum_{n=0}^{\infty} \widehat{\PP}_n(L+1-i\eta; x) \frac{w^n}{n!}, 
\]
for $L = 0, 1, 2, \ldots$, where the so called Gamow-Sommerfeld factor $C_{L}(\eta)$ is as in \eqref{Eq-Gamow-Constant} and \eqref{Eq-Alternative-Notations}. 

\end{coro}
 
Letting $x = 0$ in  Corollary \ref{Coro-Coulomb-Expansion}  gives the series expansion 
\[
    F_{L}(\eta,w) = C_{L}(\eta)\, w^{L+1} \sum_{k=L+1}^{\infty} A_{k}^{L}(\eta)\, w^{k-L-1},
\] 
where $(k-L-1)!\,A_{k}^{L}(\eta) = \widehat{\PP}_{k-L-1}(L+1-i\eta; 0)$. This expansion formula, together with the three term recurrence relation \eqref{Eq-TTRR-for-CRR-polys} satisfied by   $\{\widehat{\PP}_n(L+1-i\eta; 0)\}_{n \geq 0}$, is exactly the expansion result found in  \cite[Pg.\,538]{Abromowitz-Stegun-1972}. This expansion result for $F_{L}(\eta,w)$ has first appeared in  Yost, Wheeler and Breit \cite{Yost-etc-PhyRev1936}. Now with the hypergeometric expression in Theorem \ref{Thm-Basic-Relations} for  $\PP_{n}(b;x)$ we can give the following closed form expression for $A_{k}^{L}(\eta)$: 
\[
    A_{k+L+1}^{L}(\eta) = \frac{(-i)^{k}}{k!} \,_2F_1(-k,L+1-i\eta; 2L+2; 2),  \quad k = 0,1,2 , \ldots .
\]

One can also verify that the result corresponding to  $x = -\eta/(L+1)$ in Corollary \ref{Coro-Coulomb-Expansion} is the expansion formula presented in \cite{Meligy-NucPhy1958}.

We now give some further expansion formulas associated with the functions $\FFb(b;w)$ and the corresponding ERCW functions $\FFa(b;w)$.

\begin{theo} \label{Thm-Appell-expansion-2}
Let $b = \lambda + i \eta$, where $\lambda > 0$.  Let the real sequences $\{\aR_{n}\}_{n \geq 0} = \{\aR_{n}^{(\lambda, \eta)}\}_{n \geq 0}$ and $\{\bR_{n}\}_{n \geq 0} = \{\bR_{n}^{(\lambda, \eta)}\}_{n \geq0}$ be given by
\[
    \left[ \begin{array}l
       \aR_{n+1} \\[1ex]
       \bR_{n+1}
       \end{array} \right]
     = \frac{2}{2\lambda+n} 
    \left[ \begin{array}{cc}
       -\eta & (\lambda+n) \\[1ex]
       -(\lambda+n) & - \eta
       \end{array} \right]
    \left[ \begin{array}l
       \aR_{n} \\[1ex]
       \bR_{n}
       \end{array} \right], \quad n \geq 0,
\]
with $\aR_{0} = 1$ and $\bR_{0} = 0$. Then
\begin{equation} \label{Eq-Sine-Cosine-G}
  \begin{array}l
   \dsp \frac{1}{\mathfrak{C}(b)}  w^{-\lambda} \cos(w)\, \FFa(b;w) = \cos(w)\, \FFb(b;w) = \sum_{n=0}^{\infty} \aR_{n}\, \frac{w^n}{n!}, \\[3ex]
   \dsp  \frac{1}{\mathfrak{C}(b)}  w^{-\lambda} \sin(w)\, \FFa(b;w) = \sin(w)\, \FFb(b;w) = w \sum_{n=0}^{\infty} \frac{1}{n+1}\bR_{n+1}\, \frac{w^n}{n!}.
  \end{array}
\end{equation}
Moreover, if  $\cR_{n}(w)  = \aR_{n}\cos(w) + \frac{1}{n+1}\bR_{n+1}\, w \sin(w)$, $n \geq 0$, then 
\[
    \FFb(b;w) = \sum_{n=0}^{\infty} \cR_{n}(w)\, \frac{w^n}{n!}.
\]
\end{theo}

\begin{proof}  
In  \eqref{Eq-Appell-GenFunc},  letting $x = i$ and $x = -i$, and then, respectively,  summing and subtracting the resulting equations, we get  
\[
   \begin{array} {l}
   \dsp 2 \cos(w) \, \FFb(b;w) = \sum_{n=0}^{n} \big[ \widehat{\PP}_{n}(b; i) + \widehat{\PP}_{n}(b; -i) \big] \, \frac{w^{n}}{n!}, \\[2ex]
   \dsp i2 \sin(w) \, \FFb(b;w) = \sum_{n=0}^{n} \big[ \widehat{\PP}_{n}(b; i) - \widehat{\PP}_{n}(b; -i) \big] \, \frac{w^{n}}{n!}.
   \end{array}
\]
Hence, we set
\[
    2 \aR_{n} = \big[ \widehat{\PP}_{n}(b; i) + \widehat{\PP}_{n}(b; -i) \big] \quad \mbox{and} \quad i 2 \bR_{n} = \big[ \widehat{\PP}_{n}(b; i) - \widehat{\PP}_{n}(b; -i) \big], \quad n \geq 0. 
\]
Clearly, $\aR_0 = 1$ and $\bR_0 = 0$. To obtain the recurrence formula for $\{\aR_{n}\}_{n \geq 0}$ and $\{\bR_{n}\}_{n \geq 0}$, we observe from  \eqref{Eq-TTRR-for-CRR-polys} and  \eqref{Eq-Monic-CRR-Polys} that $\widehat{\PP}_{n}(b;i) = \overline{\widehat{\PP}_{n}(b;-i)} = \frac{{2}^{n} i^{n} (b)_{n}}{(2\lambda)_{n}}$, $n \geq 0$. Hence, 
\[
   2\aR_{n} = \frac{2^n}{(2\lambda)_{n}} \big[(i)^{n}(b)_{n} + (-i)^{n}(\overline{b})_{n}\big] \quad \mbox{and} \quad  i2\bR_{n} = \frac{2^n}{(2\lambda)_{n}} \big[(i)^{n}(b)_{n} - (-i)^{n}(\overline{b})_{n}\big], 
\]
for $n \geq 1$. Hence, all one needs to verify is that 
if 
\[
   2 \tilde{\aR}_{n} = \big[(i)^{n}(b)_{n} + (-i)^{n}(\overline{b})_{n}\big] \quad \mbox{and} \quad i2\tilde{\bR}_{n} = \big[(i)^{n}(b)_{n} - (-i)^{n}(\overline{b})_{n}\big],
\]
then there hold
\[
  \begin{array}l
    \tilde{\aR}_{n+1} = -\eta \tilde{\aR}_{n} - (\lambda+n) \tilde{\bR}_{n} \quad \mbox{and} \quad 
    \tilde{\bR}_{n+1} = -\eta \tilde{\bR}_{n} + (\lambda +n) \tilde{\aR}_{n}, \quad \mbox{for} \quad n \geq 0.
  \end{array}
\]
This is easily verified with the  substitutions $\eta = [(b+n) - (\overline{b} + n)]/(2i)$ and  $(\lambda + n) = [(b+n) + (\overline{b} + n)]/2$. This completes the proof of the formulas in \eqref{Eq-Sine-Cosine-G}. 

Now to prove the latter part of the theorem we just multiply the first formula in \eqref{Eq-Sine-Cosine-G} by $\cos(w)$ and the second formula in \eqref{Eq-Sine-Cosine-G} by $\sin(w)$, and add the resulting formulas. \hfill  
\end{proof} 

In the case of the Coulomb wave function  Theorem \ref{Thm-Appell-expansion-2} becomes: 

\begin{coro} \label{Coro-Coulomb-Expansion2} For $L \geq 0$, let  the real sequences $\{\aR_{n}\}_{n \geq 0}$ and $\{\bR_{n}\}_{n \geq 0}$ be given by $\aR_{0} = 1$, $\bR_{0} = 0$ and 
\[
    \left[ \begin{array}l
       \aR_{n+1} \\[1ex]
       \bR_{n+1}
       \end{array} \right]
     = \frac{2}{2L+n+2} 
    \left[ \begin{array}{cc}
       \eta & (L+n+1) \\[1ex]
       -(L+n+1) & \eta
       \end{array} \right]
    \left[ \begin{array}l
       \aR_{n} \\[1ex]
       \bR_{n}
       \end{array} \right], \quad n \geq 0,
\]
Then
\begin{equation} \label{Eq-Sine-Cosine-F}
  \begin{array}l
   \dsp \cos(w)\, F_{L}(\eta,w) = C_{L}(\eta)\, w^{L+1} \sum_{n=0}^{\infty} \aR_{n}\, \frac{w^n}{n!}, \\[3ex]
   \dsp  \sin(w)\, F_{L}(\eta,w)  = C_{L}(\eta)\, w^{L+2}\sum_{n=0}^{\infty} \frac{1}{n+1}\bR_{n+1}\, \frac{w^n}{n!},
  \end{array}
\end{equation}
where $C_{L}$ is as in Corollary \ref{Coro-Coulomb-Expansion}. Moreover, if $\{\cR_{n}(w)\}_{n \geq 0}$ is such that 
\[
     \cR_{n}(w)  = \aR_{n}\cos(w) + \frac{\bR_{n+1}}{n+1}\, w \sin(w), \quad n \geq 0, \]
then 
\[
    F_{L}(\eta,w) = C_{L}(\eta)\, w^{L+1}  \sum_{n=0}^{\infty} \cR_{n}(w)\, \frac{w^n}{n!}.
\]

\end{coro} 

When $\eta = 0$, the formulas for $\aR_n$ and $\bR_n$ in Theorem \ref{Thm-Appell-expansion-2} is much simpler. It is easily verified that 
\[
   \aR_1 =  \bR_0 = 0, \quad \bR_{1} = \frac{2}{2\lambda}(-\lambda) \aR_{0} = -1. 
\]
Hence, by taking $\aR_n = \aR_{n}^{(\lambda-1/2, 0)}$ and $\bR_n = \bR_{n}^{(\lambda-1/2, 0)}$ we can state:   

\begin{coro} \label{Coro-Bessel-Expansion2} For $\alpha > -1/2$, let 
\[
      \aR_{2n} = (-1)^n\frac{2^{2n}(\alpha+1/2)_{2n}}{(2\alpha+1)_{2n}}, \quad \bR_{2n+1} = (-1)^{n+1}\frac{2^{2n}(\alpha+3/2)_{2n}}{(2\alpha+2)_{2n}}, \quad n \geq 0. 
\]
Then
\begin{equation} \label{Eq-Sine-Cosine-J}
  \begin{array}l
   \dsp \cos(w)\, J_{\alpha}(w) = \frac{1}{\Gamma(\alpha+1)} \big(\frac{w}{2}\big)^{\alpha}\sum_{n=0}^{\infty} \aR_{2n}\, \frac{w^{2n}}{(2n)!}, \\[3ex]
   \dsp  \sin(w)\, J_{\alpha}(w) = \frac{w}{\Gamma(\alpha+1)} \big(\frac{w}{2}\big)^{\alpha}  \sum_{n=0}^{\infty} \frac{1}{2n+1}\bR_{2n+1}\, \frac{w^{2n}}{(2n)!}.
  \end{array}
\end{equation}
Moreover, if $\cR_{2n}(w) = \aR_{2n}\cos(w) + \frac{1}{2n+1}\bR_{2n+1}\, w \sin(w)$, $n \geq 0$, then  
\[
    J_{\alpha}(w) = \frac{1}{\Gamma(\alpha+1)} \big(\frac{w}{2}\big)^{\alpha} \sum_{n=0}^{\infty} \cR_{2n}(w)\, \frac{w^{2n}}{(2n)!}.
\]
\end{coro}

\setcounter{equation}{0}


\end{document}